\documentclass[a4paper]{amsart}

\usepackage[british,english]{babel} 
\usepackage{graphics,color,pgf}
\usepackage{epsfig}
\usepackage[ansinew]{inputenc}
\usepackage[all]{xy}
\usepackage{hyperref}
\usepackage{amssymb, amsmath,amsthm, mathtools, amscd}
\usepackage{mathrsfs}
\usepackage{stmaryrd}
\usepackage{cancel} %strike out
\usepackage{tikz}
\usepackage{tikz-cd}
\usetikzlibrary{matrix}

\theoremstyle{plain}
 \newtheorem{thm}{Theorem}
 \newtheorem{prop}[thm]{Proposition}
 
 \newtheorem{cor}[thm]{Corollary}
 \newtheorem*{op*}{Open Problem} 
 \newtheorem*{frthmA*}{Theorem A} 
 \newtheorem*{frcorB*}{Corollary B} 
  \newtheorem*{frthmAp*}{Theorem A'} 
   
\theoremstyle{definition}
 \newtheorem{exm}[thm]{Example}

\theoremstyle{remark}
 \newtheorem*{rem}{Remark}

\DeclareMathOperator{\Ker}{Ker}
\DeclareMathOperator{\Img}{Im}

\DeclareMathOperator{\alt}{\mathrm{alt}}
\DeclareMathOperator{\nalt}{\mathrm{n-alt}}

\renewcommand{\geq}{\geqslant}

\title[Alternating cochains on Furstenberg boundaries]{Alternating cochains on Furstenberg boundaries and measurable cohomology}

\author[Michelle Bucher]{%\bfseries 
Michelle Bucher} 
\address{Universit\'e de Gen\`eve}
\email{Michelle.Bucher-Karlsson@unige.ch}

\author[Alessio Savini]{%\bfseries 
Alessio Savini}
\address{Politecnico di Milano} %% Please write ful names, avoid initials
\email{alessio.savini@polimi.it}

\thanks{Supported by the Swiss National Science Foundation. 
} %% optional
\begin{document}

\begin{abstract}
In \cite{Monod}, Nicolas Monod showed that the evaluation map 
$$H^*_m(G\curvearrowright G/P)\longrightarrow H^*_m(G)$$
between the measurable cohomology of the action of a connected semisimple Lie group $G$ on its Furstenberg boundary $G/P$ and the measurable cohomology of $G$ is surjective with a non-trivial kernel in all degrees below a constant depending on $G$ and less than or equal to the rank of $G$ plus $2$. When we were looking for explicit representatives of classes in this kernel \cite{BucSav}, we were astonished to discover that some of these nontrivial classes have trivial alternation. In this paper, we refine Monod's result by identifying the non-alternating and alternating cohomology classes in this kernel. As a consequence, we show that $H^*_m(G)$ is isomorphic to the alternating measurable cohomology of $G$ acting on $G/P$ in all even degrees
$$H^{2k}_{m,\alt}(G\curvearrowright G/P)\cong H^{2k}_m(G),$$
for a majority of Lie groups, namely those for which the longest element of the Weyl group acts as $-1$ on the Lie algebra of a maximal split torus $A$ in $G$. %Furthermore, we show that the cohomology of the cocomplex of non-alternating measurable functions on the Furstenberg boundary $G/P$ is isomorphic to the invariant cohomology of  $A$ shifted by two, thus it is not trivial in general. Similarly, we show that the cohomology of the cocomplex of alternating measurable functions on $G/P$ surjects on the measurable cohomology of $G$ with a kernel given by the invariant cohomology of $A$ shifted by one. 
\end{abstract}

\maketitle

\section{Introduction}

Let $G$ be a connected semisimple Lie group with finite center. The \emph{measurable cohomology} $ H^*_m(G)$ of $G$ is defined as the cohomology of the cocomplex
$$L^0(G^{*+1},\mathbb{R})^G$$
endowed with its usual homogeneous differential. Observe that it is actually isomorphic to the continuous cohomology of $G$ by Austin and Moore \cite[Theorem A]{AM}. Given a minimal parabolic subgroup $P<G$, one can similarly consider the measurable cohomology of the action $H^*_m(G\curvearrowright G/P)$ as the cohomology of the cocomplex
$$L^0((G/P)^{*+1},\mathbb{R})^G$$
also endowed with its homogeneous differential. Notice that here we are using the usual quasi-invariant measure class on the quotient $G/P$. The evaluation on a base point induces a map
\begin{equation*}%\label{eval map}
ev:H^*_m(G\curvearrowright G/P)\longrightarrow H^*_m(G)\end{equation*}
which is easily shown not to depend on the base point. Monod remarkably showed that this evaluation map is surjective. Furthermore, he gives a precise description of the kernel in terms of the invariant cohomology of a maximal split torus $A<P$. More precisely, if $w_0$ is a representative of the longest element in the Weyl group associated to $A$, then $w_0$ acts on $H^\ast_m(A)$ via the adjoint representation and we define the $w_0$-invariant cohomology $H^\ast_m(A)^{w_0}$ as those classes fixed by $\mathrm{Ad}(w_0)$. 

\begin{thm}\cite[Theorem B]{Monod}\label{Thm Monod}\label{thm monod} Let $G$ be a connected semisimple Lie group with finite center. Let $P$ be a minimal parabolic subgroup, $A < P$ a maximal split torus and $w_0$ a representative of the longest element in the Weyl group associated to $A$. The evaluation map 
$$ H^*_m(G\curvearrowright G/P)\longrightarrow H^*_m(G)$$
is surjective and its kernel 
$$NH^*_m(G\curvearrowright G/P):= \mathrm{Ker}(H^*_m(G\curvearrowright G/P)\longrightarrow H^*_m(G))$$
fits into an exact sequence
\begin{equation}\label{SES}0 \longrightarrow H_m^{p-2}(A)^{w_0} \longrightarrow NH^p_m(G\curvearrowright G/P) \longrightarrow  H_m^{p-1}(A)^{w_0} \longrightarrow 0,\end{equation}
for $p\geq 2$. \end{thm}

Observe that Theorem \ref{thm monod} is not formulated correctly in \cite{Monod}. Indeed, it is stated there under the hidden  assumption that $w_0$ acts as $-1$ on $\mathfrak{a}$ (which is often true, but wrong for $\mathrm{SL}(n,\mathbb{R})$ when $n\geq 3$, for example). While the statement takes a somehow different form, the adaptation to a general $w_0$-action has essentially no effect on Monod's proof. %We will recover Theorem \ref{thm monod}  in full generality as a direct consequence of our Theorems \ref{thm negative complex} and  \ref{thm positive complex}.

An immediate corollary of Theorem \ref{thm monod}  is that $H^p_m(G\curvearrowright G/P)\cong H^p_m(G)$ for $p$ strictly greater than  the rank of $G$ plus $2$. We will see in Corollary \ref{cor w0 -1} a refinement of this isomorphism, when considering alternating cocycles, valid in all even degrees for the majority of semisimple Lie groups.

%\subsection*{Alternation}

The symmetric group $\mathrm{Sym}(p+1)$ naturally acts on $(G/P)^{p+1}$ and this action commutes with the diagonal action of $G$ on $(G/P)^{p+1}$. As a consequence we can define an \emph{alternation} function
$$\mathrm{Alt}_{p+1}: L^0((G/P)^{p+1},\mathbb{R})^G\longrightarrow L^0((G/P)^{p+1},\mathbb{R})^G$$
by
$$\mathrm{Alt}_{p+1}(f)(x_0,\dots,x_p)=\frac{1}{(p+1)!} \sum_{\sigma\in \mathrm{Sym}(p+1)} \mathrm{sign}(\sigma) f (x_{\sigma(0)},\dots,x_{\sigma(p)}).$$
We define the \emph{non-alternating}, respectively \emph{alternating}, cochains by 
$$L^0_{\nalt}((G/P)^{p+1})^G:=\mathrm{Ker}(\mathrm{Alt}_{p+1}), \quad \mathrm{respectively} \quad L^0_{\alt}((G/P)^{p+1})^G:=\mathrm{Im}(\mathrm{Alt}_{p+1}).$$
Since $\mathrm{Alt}_{p+1}$ is idempotent, there is  a direct sum decomposition
$$L^0((G/P)^{p+1})=L^0_{\nalt}((G/P)^{p+1})\oplus L^0_{\alt}((G/P)^{p+1}).$$
Furthermore, a  measurable cochain $f$ on $(G/P)^{p+1}$ is alternating if and only if it satisfies 
$$
f(x_{\sigma(0)},\ldots,x_{\sigma(p)})=\mathrm{sign}(\sigma)f(x_0,\ldots,x_p) . 
$$
for any permutation $\sigma \in \mathrm{Sym}(p+1)$.

It is straightforward to check that the alternation commutes with the homogeneous coboundary operator and thus restricts to both the non-alternating and alternating cochains. Denoting by $H^p_{m,\nalt}(G \curvearrowright G/P)$ and $H^{p}_{m,\alt}(G \curvearrowright G/P)$ their respective cohomologies, we immediately obtain the direct sum decomposition
$$
H^*_m(G\curvearrowright G/P)=H^*_{m,\nalt}(G \curvearrowright G/P)\oplus H^{*}_{m,\alt}(G \curvearrowright G/P),
$$
which restricts to a direct decomposition of the kernel
\begin{equation}\label{eq:decomposition:kernel:boundary}
NH^*_m(G\curvearrowright G/P)\cong NH^*_{m,\nalt}(G\curvearrowright G/P)\oplus NH^*_{m,\alt}(G\curvearrowright G/P).
\end{equation}

In the case of groups, the alternation map is homotopic to the identity, thus we have that $H^*_{m,\alt}(G)=H^*_m(G)$. One could expect that the same holds also for the cohomology of the boundary, but this is not the case. The problem here is that the homotopy between the alternation and the identity involves evaluating a cochain on tuples of nondistinct points (namely the prism between a simplex and its alternation), which is a set of measure zero, so that the evaluation makes no sense for $L^0$-cochains. The fact that $H^*_{m,\alt}(G,\mathbb{R})=H^*_m(G,\mathbb{R})$ implies immediately that 
$$H^*_{m,\nalt}(G\curvearrowright G)=NH^*_{m,\nalt}(G\curvearrowright G)$$
is included in the kernel. What is more surprising is that the decomposition \eqref{eq:decomposition:kernel:boundary} precisely corresponds to the short exact sequence (\ref{SES}). Indeed, if we compute the cohomology of the cocomplex of non-alternating measurable functions we get back the $w_0$-invariant cohomology of the maximal split torus $A$ with a degree shifted by $2$, which is the kernel of the short exact sequence (\ref{SES}): 

%By looking at Theorem \ref{thm monod} can ask what is the main difference between the role played by the invariant cohomology of $A$ in degree $p-1$ with respect to the one of degree $p-2$. One of the aim of this paper is to clarify the main difference between them. 

\begin{thm}\label{thm negative complex}
Let $G$  be a connected semisimple Lie group with finite center. Let $P$ be a minimal parabolic subgroup, $A < P$ a maximal split torus and $w_0$ a representative of the longest element in the Weyl group associated to $A$. Then
\begin{itemize}
\item $H^p_{m,\nalt}(G \curvearrowright G/P) \cong 0$ for $p=0,1,2$;
\item $H^{p}_{m,\nalt}(G \curvearrowright G/P) \cong H^{p-2}_m(A)^{w_0}$ for $p \geq 3$. 
\end{itemize}
\end{thm}

%An equivalent formulation of the theorem above is that the kernel of the short exact sequence of Equation \eqref{eq ses monod} corresponds precisely to the non-alternating cohomology. 

In a similar way, we compute explicitly the cohomology of the alternating cocomplex. 

\begin{thm}\label{thm positive complex}
Let $G$ be a connected semisimple Lie group with finite center. Let $P$ be a minimal parabolic subgroup, $A < P$ a maximal split torus and $w_0$ a representative of the longest element in the Weyl group associated to $A$. There is a short exact sequence 
\begin{equation}\label{eq ses positive boundary}
\xymatrix{
0 \ar[r] & H^{p-1}_m(A)^{w_0} \ar[r] & H^p_{m,\alt}(G \curvearrowright G/P)\ar[r] & H^p_m(G) \ar[r] & 0,
}
\end{equation}
for $p \geq 2$. 
\end{thm}

Observe that for $p=0,1$, we have $H^0_{m,\alt}(G \curvearrowright G/P)\cong \mathbb{R}$ and $H^1_{m,\alt}(G \curvearrowright G/P)=0$ by $2$-transitivity on generic points in $G/P$. In particular $H^p_{m,\alt}(G \curvearrowright G/P)\cong H^p_m(G)$ for $p=0,1$. 

%In this case the cohomology of the alternating part has two contribution, the first one coming from the measurable cohomology of $G$, whereas the second one coincides with the image of the short exact sequence of Equation \eqref{eq ses monod}. Equivalently the cocycles lying the kernel of Monod' short exact sequence can be represented by alternating cochains. 

A direct consequence of our Theorem \ref{thm positive complex}, is that for a large class of groups, we can actually compute the measurable cohomology of $G$ using alternating cochains on the associated Furstenberg boundary $G/P$ in even degree. Indeed, if $w_0$ acts as $-1$ on the Lie algebra of the maximal split torus $A$, it is immediate that 
$$
H^k_m(A)^{w_0}=
\begin{cases*}
H^k_m(A),& \ \ \textup{if $k$ is even}, \\
0,& \ \ \textup{if $k$ is odd}. 
\end{cases*}
$$
As a consequence, the left term of the short exact sequence  \eqref{eq ses positive boundary} vanishes when $p=k+1$ is even and we immediately obtain:

\begin{cor}\label{cor w0 -1}
Let $G$ be a connected semisimple Lie group with finite center. Let $P$ be a minimal parabolic subgroup, $A < P$ a maximal split torus and $w_0$ a representative of the longest element in the Weyl group associated to $A$. If $w_0$ acts as $-1$ on the Lie algebra of $A$, then for every $k \geq 0$ we have that 
$$
H^{2k}_m(G) \cong H^{2k}_{m,\mathrm{alt}}(G \curvearrowright G/P).
$$
\end{cor}

\begin{rem}
Recall that $w_0$ acts as $-1$ for the groups of type $B_n,C_n,E_7,E_8,G_2,D_{2n}$ \cite[Plates I-IX]{BouVI}, and hence also for their products. So our statement is true for all semisimple Lie groups with no factor of type $A_n,D_{2n+1},E_6$.
\end{rem}

Note that Theorem \ref{thm positive complex} also shows that Corollary \ref{cor w0 -1} has no chance to hold in general, that is in degree $p$ whenever $H^{p-1}_m(A)^{w_0}$ is not trivial. 

Explicit examples of both non-alternating and alternating cocycles on $G/P$ representing non-trivial classes in the kernel of the evaluation map are presented in  \cite{BucSav} in low degrees for products of isometry groups of hyperbolic spaces and for $\mathrm{SL}(3,\mathbb{K})$, where $\mathbb{K}=\mathbb{R},\mathbb{C}$. As an illustration of  our Theorems \ref{thm negative complex} and \ref{thm positive complex} we reproduce in Section \ref{Section: examples} the two simplest such examples, namely an alternating  cocycle in degree $3$ and  a non-alternating in degree $4$ giving rise to nontrivial cohomology classes in $NH^*_m(G\curvearrowright G/P)$ for a product of two groups of isometries of real hyperbolic spaces.  %For instance, when $G=\mathrm{SL}(3,\mathbb{K})$, a measurable $G$-invariant alternating cocycle in degree $2$ is given by the logarithm of the modulus of the \emph{triple ratio} associated to a \emph{generic} configuration of full flags in $\mathbb{K}^3$ \cite[Theorem 5]{BucSav}. In a similar way, a measurable $G$-invariant non-alternating cocycle in degree $3$ is given by the difference between the logarithms of the moduli of suitable \emph{cross ratios} associated to a generic configuration \cite[Theorem 6]{BucSav}. 

The proofs of Theorems \ref{thm negative complex} and  \ref{thm positive complex} are a refinement of Monod's proof of Theorem \ref{thm monod}. Indeed, we decompose the spectral sequences associated to the bicomplex considered by Monod into the direct sum of its non-alternated and alternated subspaces. This leads to two bicomplexes and their associated spectral sequences naturally lead to the strengthening of Monod's Theorem  \ref{thm monod} given by our Theorems \ref{thm negative complex} and  \ref{thm positive complex}.  %is based on the machinery of spectral sequences. In both cases we introduce a spectral sequence which is analogous to the one used in Monod's proof of Theorem \ref{thm monod}. Our spectral sequences are finer enough to keep track about the alternating property as we reach the convergence. 

%odd degree: in fact in \cite{BucSav} we exhibit an alternating cocycle representing a non-trivial class in the kernel $NH^3_m(G \curvearrowright G/P)$ when $G$ is the non-trivial product of isometry groups of hyperbolic spaces. 

%\begin{proof}[Proof of Theorem \ref{thm w0 -1}] By Theorem \ref{thm positive complex} we know that, for $p \geq 2$, we have the following short exact sequence 
%\begin{equation}\label{SES alternating}
%\xymatrix{
%0 \ar[r] & H^{p-1}_m(A)^{w_0} \ar[r] & H^p_{m,\alt}(G \curvearrowright G/P)\ar[r] & H^p_m(G) \ar[r] & 0,
%}
%\end{equation}
%If we now suppose that $w_0$ acts as $-1$ on the Lie algebra $\mathfrak{a}=\mathrm{Lie}(A)$, then we have that 
%$$
%H^p_m(A)^{w_0}=
%\begin{cases*}
%H^p_m(A),& \ \ \textup{if $p$ is even}, \\
%0,& \ \ \textup{if $p$ is odd}. 
%\end{cases*}
%$$
%As a consequence, for $p=2k$ the short exact sequence of Equation \ref{SES alternating} implies that 
%$$
%H^{2k}_{m,\alt}(G \curvearrowright G/P) \cong H^{2k}_m(G),
%$$
%and the statement is proved. 
%\end{proof}

\subsection*{Acknowledgements} We are grateful to Nicolas Monod for several useful discussions in the preparation of this paper.

\section{Explicit non-alternating and alternating classes in $NH^*_m(G\curvearrowright G/P) $}\label{section examples}\label{Section: examples} 
As an illustration, we present here two nontrivial classes in $NH^*_m(G\curvearrowright G/P)$ that we found in \cite{BucSav} navigating through Monod' spectral sequence. The cocycle in degree $3$ is obviously alternating and represents a nontrivial class in the kernel $NH^3_m(G\curvearrowright G/P)$. The one in degree $4$ is also easily shown to be a non-alternating cocycle. On the other hand, we needed the construction via the spectral sequence to show that it defines a nontrivial class in $NH^4_m(G\curvearrowright G/P)$. We refer the reader to \cite{BucSav} for the details. Here we only exhibit the cocycles and show why the one in degree $4$ is non-alternating. Although we write the cocycles in the case of two factors in order to simplify the notation, the generalization to several factors is straightforward. 

For $G=\mathrm{Isom}^+(\mathbb{H}^n)$, we identify $\partial \mathbb{H}^n$ with  $\mathbb{R}^{n-1}\cup \{{\infty} \}$, choose $P=\mathrm{Stab}({\infty})$ and take as maximal abelian subgroup $A<P$ the group of homotheties of the upper half space model, namely $A:=\{a_\lambda:x\mapsto \lambda x\mid \lambda\in \mathbb{R}_{>0}\}$. For any distinct points $x_0,\dots,x_3\in \mathbb{R}^{n-1}\cup \{{\infty} \}=\partial \mathbb{H}^n$, we define their \emph{(positive) cross ratio}  by 
\begin{equation*}\label{eq:cross:ratio:real}
{b}(x_0,\cdots,x_3):=\frac{\|x_2 -x_0\| \| x_3 - x_1\|}{\|x_2-x_1\|\|x_3-x_0\|}\in \mathbb{R}_{>0}, 
\end{equation*}
where $\| \cdot \|$ denotes the Euclidean norm and with the usual convention that ${\infty}/{\infty}=1$. This is precisely the absolute value, respectively the modulus, of the classic cross ratio of 4 points on $\partial \mathbb{H}^2=P^1\mathbb{R}$ and $\partial \mathbb{H}^3=P^1\mathbb{C}$. It is clear that this cross ratio is  invariant under the isometry group of $\mathbb{H}^n$. 
 
Let now $G=\mathrm{Isom}^+(\mathbb{H}^n)\times \mathrm{Isom}^+(\mathbb{H}^m)$, where $n,m\geq 2$. Then a maximal split torus $A$ has rank $2$, with $w_0$ acting as $-1$ on its Lie algebra. As a consequence
$$H^k_m(A)^{w_0}=\left\{ \begin{array}{ll}
\mathbb{R}& \mathrm{if \ } k=0 \mathrm{ \ or \ } 2,\\
0&\mathrm{otherwise}
\end{array}\right. $$
and our Theorems \ref{thm negative complex} and \ref{thm positive complex} predict that
$$NH^k_m(G\curvearrowright G/P)=\left\{ \begin{array}{ll}
NH^k_{m,\mathrm{alt}}(G\curvearrowright G/P)\cong  \mathbb{R}& \mathrm{if \ } k=3,\\
NH^k_{m,\mathrm{n-alt}}(G\curvearrowright G/P)\cong \mathbb{R}& \mathrm{if \ } k=4,\\
0&\mathrm{otherwise}.
\end{array}\right. $$

\begin{exm}\cite[Theorem 2]{BucSav} A generator of $NH^3_{m,\mathrm{alt}}(G\curvearrowright G/P)$ is given by the $G$-invariant alternating cocycle
$$\begin{array}{rcl}
(\partial \mathbb{H}^n \times \partial \mathbb{H}^m)^4&\longrightarrow &\mathbb{R}\\
((x_0,y_0),\dots,(x_3,y_3))&\longmapsto & \mathrm{det}\left( \begin{array}{cc}
\mathrm{log}(b(x_0,x_1,x_2,x_3))&\mathrm{log}(b(x_1,x_2,x_3,x_0))\\
\mathrm{log}(b(y_0,y_1,y_2,y_3))&\mathrm{log}(b(y_1,y_2,y_3,y_0))
\end{array} \right) . \end{array}$$

The fact that this is an alternating cocycle is straightforward from the properties of the cross ratio. 
\end{exm}

%\begin{proof} It is straightforward to check that this is indeed a $G$-invariant alternating cocycle. Now observe that by the transitivity of $G$ on triples of generic triple of points, any alternating coboundary is trivial. In particular, being a nonzero coycle, it represents a non-trivial class in  $H^3_{\mathrm{alt}}(G\curvearrowright G/P)$. Furthermore, it lies trivially in the kernel of the evaluation map since $H^3_m(G)=0$. Indeed, $H^*_m(G)$ is generated (as an algebra) by the cup products of an even number of volume classes. It is in particular only nontrivial in degree $0$ and $n+m\geq 4$. \end{proof}

%\begin{cor} \label{Gap for products} Conjecture \ref{Conj Gap} is true for $G=\mathrm{Isom}(\mathbb{H}^n)\times \mathrm{Isom}(\mathbb{H}^m)$.
%\end{cor}

%\begin{proof} Any class in $NH^*_{m,\mathrm{alt}}(G\curvearrowright G/P)$ is a multiple of the above cocycle, which is clearly not bounded. This proves the corollary since there are no $G$-invariant alternating coboundaries in degree $3$ by the transitivity of $G$ on generic triples of points. 
%\end{proof}

\begin{exm}\cite[Theorem 4]{BucSav} A generator of $NH^4_{m,\mathrm{n-alt}}(G\curvearrowright G/P)$ is given by the $G$-invariant non-alternating cocycle
$$\begin{array}{rcl}
(\partial \mathbb{H}^n \times \partial \mathbb{H}^m)^5&\longrightarrow &\mathbb{R}\\
((x_0,y_0),\dots,(x_4,y_4))&\longmapsto & \mathrm{det}\left( \begin{array}{cc}
\mathrm{log}(b(x_0,x_1,x_2,x_3))&\mathrm{log}(b(x_1,x_2,x_3,x_4))\\
\mathrm{log}(b(y_0,y_1,y_2,y_3))&\mathrm{log}(b(y_1,y_2,y_3,y_4))
\end{array} \right) . \end{array}$$

Let us show that it is non-alternating. In order to do so, we will prove that the cocycle applied to the 5-tuple taken in reverse order, $((x_4,y_4),\dots,(x_0,y_0))$, changes by a factor $-1$, whereas the permutation $(0,4)(1,3)$ has sign $+1$. This easily implies that its alternation vanishes. We evaluate our cocycle on $((x_4,y_4),\dots,(x_0,y_0))$ to obtain
$$ \mathrm{det}\left( \begin{array}{cc}
\mathrm{log}(b(x_4,x_3,x_2,x_1))&\mathrm{log}(b(x_3,x_2,x_1,x_0))\\
\mathrm{log}(b(y_4,y_3,y_2,y_1))&\mathrm{log}(b(y_3,y_2,y_1,y_0))
\end{array} \right).$$
Since $b(z_1,z_2,z_3,z_4)=b(z_4,z_3,z_2,z_1)$, this determinant is equal to
$$ \mathrm{det}\left( \begin{array}{cc}
\mathrm{log}(b(x_1,x_2,x_3,x_4))&  \mathrm{log}(b(x_0,x_1,x_2,x_3))\\
\mathrm{log}(b(y_1,y_2,y_3,y_4))  &\mathrm{log}(b(y_0,y_1,y_2,y_3))
\end{array} \right) ,$$
which is the evaluation on the tuple $((x_0,y_0),\dots,(x_4,y_4))$ where the columns are exchanged, hence the factor $-1$.
\end{exm}

%To check that this cocycle is nontrivial, it is necessary to go back to its construction, by expliciting the injection in the short exact sequence (\ref{SES}) at the level of cochains. For this we refer the reader to the detailed computations carried on in \cite{BucSav}.  

%We do not see any explanation for the very similar expressions for the above classes in degree $3$ and $4$, which differ only in the last entry in the second column. 

The discovery of this non-alternating cocycle in degree $4$ was the main motivation of the present paper. Further cocycles in the kernel are exhibited for $\mathrm{SL}(3,\mathbb{K})$ in \cite{BucSav} when $\mathbb{K}$ is $\mathbb{R}$ or $\mathbb{C}$. They are alternating in degree $2$ and non-alternating in degree $3$.

\section{The bicomplex of (non-)alternating functions on the Furstenberg boundary}\label{sec spectral sequence alternating}

Let $G$ be a connected semisimple Lie group with finite center. Let $P<G$ be a minimal parabolic subgroup and consider $A <P$ a maximal split torus. We denote by $w_0$ a representative of the longest element in the Weyl group associated to $A$. 

To establish Theorem \ref{thm monod} \cite[Theorem B]{Monod}, Monod considers the spectral sequences associated to the bicomplex 
$$
C^{p,q}:=L^0(G^{p+1},L^0((G/P)^q))^G,
$$
where the vertical differential 
$$
d^\uparrow:C^{p,q} \rightarrow C^{p+1,q}
$$
coincides with the usual homogeneous differential on $G$, whereas the horizontal differential 
$$
d^\rightarrow:C^{p,q} \rightarrow C^{p,q+1}
$$
is the homogeneous differential on $G/P$ multiplied by $(-1)^{p+1}$ to ensure that $d^\rightarrow d^\uparrow=d^\uparrow d^\rightarrow$. The decomposition
$$L^0((G/P)^q) \cong L^0_{\nalt}((G/P)^q) \oplus L^0_{\alt}((G/P)^q)$$
allows us to decompose the above bicomplex, where for $q=0$ we just set $L^0_{\nalt}((G/P)^0)=0$ and $L^0_{\alt }((G/P)^0)= \mathbb{R}$. Indeed, for a subscript  $\varepsilon \in \{\alt, \nalt\}$, set
$$C^{p,q}_\varepsilon:=L^0(G^{p+1},L^0_{\varepsilon}((G/P)^q))^G,
$$
so that
\begin{equation}\label{eq decomposition ss}
C^{p,q}=C^{p,q}_{\nalt} \oplus C^{p,q}_{\alt} .
\end{equation}
The vertical and the horizontal differentials clearly restrict to both $C^{p,q}_{\nalt}$ and $C^{p,q}_{\alt} $. We will thus consider the spectral sequences associated to these bicomplexes. Beside giving the refinement of Monod's Theorem \ref{thm monod}, it also slightly simplifies Monod's proof as the associated spectral sequences  are somehow easier to handle. For example, for the second spectral sequences $^{II}_\varepsilon E$ at most two columns remain non-trivial on the second page, whereas three columns remain non-trivial on the second page for the original bicomplex.

For a given $\varepsilon \in \{\alt,\nalt\}$, we naturally obtain two different spectral sequences. The first page of the first spectral sequence is obtained as follows
\begin{equation}\label{eq first page I}
^{I}_\varepsilon E_1^{p,q}:=(H^q(C^{p,q}_\varepsilon,d^\rightarrow),d_1=d^\uparrow). 
\end{equation}

\begin{prop}\label{prop first ss degenerates}
For $\varepsilon \in \{ \alt, \nalt \}$, the spectral sequence $^{I}_\varepsilon E$ degenerates immediately, i.e. $^{I}_\varepsilon E_1^{p,q}=0$ for every $p,q\geq 0$. 
\end{prop}

\begin{proof}
In virtue of the decomposition of Equation \ref{eq decomposition ss}, for every $p,q \geq 0$ we can write 
$$
H^q(C^{p,q},d^\rightarrow) \cong H^q(C_{\nalt}^{p,q},d^\rightarrow) \oplus H^q(C_{\alt}^{p,q},d^\rightarrow). 
$$
By \cite[Proposition 6.1]{Monod} we know that the left-hand side is trivial. %Monod's proof is actually given for connected groups, but it still holds without any change also in the non-connected case. 
As a consequence, we must have that
$$
H^q(C_{\alt}^{p,q},d^\rightarrow) \cong H^q(C_{\nalt}^{p,q},d^\rightarrow) \cong 0,
$$
and the statement is proved. 
\end{proof}

The fact that the spectral sequence $^I_\varepsilon  E_1^{p,q}$ degenerates (immediately) automatically implies that  the second spectral sequence associated to the bicomplex $C^{p,q}_\varepsilon$, namely
$$
^{II}_\varepsilon  E_1^{p,q}:=(H^p(C^{p,q}_\varepsilon,d^\uparrow),d_1=d^\rightarrow),
$$
also degenerates \cite[Proposition A.9]{Guichardet}, but not immediately. We will exploit this to establish the isomorphisms in Theorems \ref{thm negative complex} and  \ref{thm positive complex}. 

The $(p,q)$-th entry of the first page of this spectral sequence is given by 
$$
^{II}_\varepsilon  E_1^{p,q}:=H^p_m(G,L^0_\varepsilon((G/P)^q)).
$$%For sake of simiplicity we are going to assume that $q \geq 2$ until the end of the section. 
Recall that  we have a direct sum decomposition 
$$
L^0((G/P)^q) \cong L^0_{\nalt} ((G/P)^q) \oplus L^0_{\alt}((G/P)^q),
$$
%Equivalently, the short exact sequence 
%\begin{equation} \label{eq split exact}
%0 \rightarrow  L^0_{\nalt}((G/P)^q) \rightarrow  L^0((G/P)^q) \rightarrow  L^0_{\alt}((G/P)^q) \rightarrow 0
%\end{equation}
%is split exact. The existence of the long exact sequence for coefficients in measurable cohomology \cite[Proposition 21]{Moore} implies that we have a long exact sequence
%$$
%\begin{tikzcd}
%  & \dots  \rar
%             \ar[draw=none]{d}[name=X, anchor=center]{}
%    &   H^{p-1}_m(G,L^0_{\alt}(G/P)^q))  \ar[rounded corners,
%            to path={ -- ([xshift=2ex]\tikztostart.east)
%                      |- (X.center) \tikztonodes
%                      -| ([xshift=-2ex]\tikztotarget.west)
%                      -- (\tikztotarget)}]{dll}[at end]{ }\\  
%H^p_m(G,L^0_{\nalt}(G/P)^q))  \rar &  H^p_m(G,L^0(G/P)^q)) \rar
%             \ar[draw=none]{d}[name=X, anchor=center]{}
%    &   H^p_m(G,L^0_{\alt}(G/P)^q))\ar[rounded corners,
%            to path={ -- ([xshift=2ex]\tikztostart.east)
%                      |- (X.center) \tikztonodes
%                      -| ([xshift=-2ex]\tikztotarget.west)
%                      -- (\tikztotarget)}]{dll}[at end]{ }\\      
%H^{p+1}_m(G,L^0_{\nalt}(G/P)^q))   \rar & \dots &
%\end{tikzcd}
%$$
%The fact that the sequence in \eqref{eq split exact} is split exact implies that the transgression maps in the long exact sequence are identically zero. Equivalently, we have a decomposition
which implies an analogous decomposition at the level of coefficients, namely
\begin{equation}\label{eq split exact cohomology}
H^p_m(G,L^0((G/P)^q)) \cong H^p_m(G,L^0_{\nalt}((G/P)^q)) \oplus H^p_m(G,L^0_{\alt}(G/P)^q)).
\end{equation}
Furthermore, alternation in the coefficients induces an alternation map
$$\mathrm{Alt}_q: H^p_m(G,L^0((G/P)^q)) \longrightarrow H^p_m(G,L^0((G/P)^q)) $$
and it is immediate to check that
\begin{equation}\label{ H alt Ker Im} H^p_m(G,L^0_{\nalt}((G/P)^q))\cong \mathrm{Ker}(\mathrm{Alt}_q)\quad \mathrm{and}\quad H^p_m(G,L^0_{\alt}((G/P)^q))\cong \mathrm{Im}(\mathrm{Alt}_q).\end{equation}

The next goal of our investigation is to establish an explicit isomorphism for the column $^{II}_\varepsilon E_1^{p,2}$, namely where we consider two points on the Furstenberg boundary $G/P$. Recall that there exists a natural action of $w_0$ on the measurable cohomology $H^p_m(A)$ induced by the adjoint action. In particular, we can consider the map 
$$
\Pi^p_{w_0}:H^p_m(A) \rightarrow H^p_m(A), \  \ \ \alpha \mapsto \alpha-\mathrm{Ad}(w_0)(\alpha). 
$$
We say that an element $\alpha \in H^p_m(A)$ is $w_0$-\emph{invariant} if it lies in the kernel of $\Pi^p_{w_0}$, whereas we say that $\alpha$ is $w_0$-\emph{equivariant} if it lies in the image of $\Pi^p_{w_0}$. We denote by 
$$
H^p_m(A)^{w_0}:=\Ker(\Pi^p_{w_0})
$$
the subspace of $w_0$-invariant classes. 

\begin{prop}\label{prop w0 invariant}
We have the following isomorphisms
\begin{eqnarray*}
H^p_m(G,L^0_{\nalt}((G/P)^2)) &\cong &\Ker(\Pi^p_{w_0}), \\
H^p_m(G,L^0_{\alt}((G/P)^2)) &\cong& \Img(\Pi^p_{w_0}) , 
\end{eqnarray*}
for every $p \geq 0$. 
\end{prop}

\begin{proof}
Let $M$ be the centralizer of the Lie algebra associated to $A$ in the maximal compact subgroup of $G$. There exists a $G$-orbit of full measure in $(G/P)^2$ which is isomorphic to the quotient $G/MA$ via the map $gMA \mapsto (gP,gw_0P)$. For every $p \geq 0$, the latter identification allows us to write the following isomorphism
$$
H^p_m(G,L^0((G/P)^2)) \cong H^p_m(G,L^0(G/MA)),
$$
which actually holds at the level of cochains
\begin{equation}\label{eq modules two points}
L^0(G^{p+1},L^0(G/P)^2)^G \cong L^0(G^{p+1},L^0(G/MA))^G. 
\end{equation}

We remind the reader that the measurable cohomology $H^p_m(L)$ of a closed subgroup $L<G$ can be computed using the cocomplex $(L^0(G^{\ast+1})^L,d^\ast)$. In a similar way, if $L$ and $T$ are closed subgroups of $G$ such that $L<T$, we implement the \emph{cohomological restriction} via the inclusion of invariants
$$
L^0(G^{p+1})^{T} \rightarrow L^0(G^{p+1})^{L}.
$$

Thanks to the previous observation, for every $p \geq 0$, we can write the map inducing the Eckmann-Shapiro isomorphism \cite[Theorem 6]{Moore} as follows 
$$
\Psi^p:L^0(G^{p+1})^{MA} \rightarrow L^0(G^{p+1},L^0(G/MA))^G, 
$$
$$
 \Psi^p(f)(g_0,\ldots,g_p)(hMA):=f(h^{-1}g_0,\ldots,h^{-1}g_p). 
$$
At a cohomological level we obtain a chain of isomorphisms
$$
H^p_m(G,L^0(G/MA)) \cong H^p_m(MA) \cong H^p_m(A),
$$
where the last one is precisely the restriction map. The latter is an isomorphism because $M$ is compact and centralizes $A$ \cite[Theorem 9.1]{Blanc}. We deduce that 
\begin{equation}\label{eq MA isom GP2}
H^p_m(MA) \cong H^p_m(G,L^0((G/P)^2)) 
\end{equation}
and the above isomorphism is obtained by composing the map $\Psi^p$ and the identification of Equation \eqref{eq modules two points}, namely
$$
\Phi^p:L^0(G^{p+1})^{MA} \rightarrow L^0(G^{p+1},L^0((G/P)^2))^G, 
$$
$$
\Phi^p(f)(g_0,\ldots,g_p)(hP,hw_0P)=f(h^{-1}g_0,\ldots,h^{-1}g_p). 
$$
The main point in the proof is to check that the above map intertwines the operator $\Pi^p_{w_0}$ and the operator $2\mathrm{Alt}_2$  acting on the coefficients. More precisely, we claim that the diagram
\begin{equation}\label{eq intertwine projections}
\xymatrix{
L^0(G^{p+1})^{MA} \ar[rr]^{\hspace{-15pt}\Phi^p} \ar[d]^{\Pi^p_{w_0}} && L^0(G^{p+1},L^0((G/P)^2)) \ar[d]^{2\mathrm{Alt}_2} \\
L^0(G^{p+1})^{MA} \ar[rr]^{\hspace{-15pt}\Phi^p} &&  L^0(G^{p+1},L^0((G/P)^2))
}
\end{equation}
induces a commutative diagram in cohomology. Let $f \in L^0(G^{p+1})^{MA}$. The operator $\Pi^p_{w_0}$ acts on the space $L^0(G^{p+1})^A$ as follows
$$
(\Pi^p_{w_0}f)(g_0,\ldots,g_p)=f(g_0,\ldots,g_p)-f(w_0^{-1}g_0w_0,\ldots,w_0^{-1}g_0w_0).
$$
By \cite[Chapter I, Section 7]{Guichardet}, the right multiplication by $w_0$ is homotopic to the identity. Thus, up to homotopy, we can suppose that the action of the operator $\Pi^p_{w_0}$ is actually given by 
$$
(\Pi^p_{w_0}f)(g_0,\ldots,g_p)=f(g_0,\ldots,g_p)-f(w_0^{-1}g_0,\ldots,w_0^{-1}g_0).
$$
If we now apply the map $\Phi^p$ we obtain
\begin{align*}
&(\Phi^p \circ \Pi^p_{w_0})(f)(g_0,\ldots,g_p)(hP,hw_0P)\\
=&(\Pi^p_{w_0}f)(h^{-1}g_0,\ldots,h^{-1}g_p)\\
=&f(h^{-1}g_0,\ldots,h^{-1}g_p)-f(w_0^{-1}h^{-1}g_0,\ldots,w_0^{-1}h^{-1}g_p).
\end{align*}
In a similar way we can consider the other composition 
\begin{align*}
&(2\mathrm{Alt}_2 \circ \Phi^p)(f)(g_0,\ldots,g_p)(hP,hw_0P)\\
=&\Phi^p(f)(g_0,\ldots,g_p)(hP,hw_0P)-\Phi^p(f)(g_0,\ldots,g_p)(hw_0P,hP)\\
=&f(h^{-1}g_0,\ldots,h^{-1}g_p)-f(w_0^{-1}h^{-1}g_0,\ldots,w_0^{-1}h^{-1}g_p),
\end{align*}
where we exploited the fact that $w_0^2$ is an element of $M<P$ to move from the second line to the third one. The diagram \eqref{eq intertwine projections} thus induces the following commutative diagram
$$
\xymatrix{
H^p_m(A) \ar[rr]^{\hspace{-25pt}\cong} \ar[d]^{\Pi^p_{w_0}} && H^p_m(G,L^0((G/P)^2)) \ar[d]^{2\mathrm{Alt}_2} \\
H^p_m(A) \ar[rr]^{\hspace{-25pt}\cong} && H^p_m(G,L^0((G/P)^2)),
}
$$
where, with an abuse of notation, we used the same letters for the maps induced at cohomological level and we exploited the restriction isomorphism to pass from $MA$ to $A$. In particular, the Eckmann-Shapiro induction determines the following isomorphisms
$$
\Ker(\Pi^p_{w_0}) \cong \Ker(2\mathrm{Alt}_2)\cong \Ker(\mathrm{Alt}_2), \ \ \ \Img(\Pi^p_{w_0}) \cong \Img(2\mathrm{Alt}_2) \cong \Img(\mathrm{Alt}_2),
$$
which in view of the isomorphisms in (\ref{ H alt Ker Im}) finishes the proof of the proposition.  \end{proof}

We conclude the section with a vanishing result for the columns of the page $^{II}_\varepsilon E_1^{p,q}$ when $q \geq 3$. 

\begin{prop}\label{prop zero columns}
Let $\varepsilon \in \{\alt, \nalt\}$. Then the measurable cohomology 
$$
H^p_m(G,L^0_\varepsilon((G/P)^q)) \cong 0
$$
vanishes when $p \geq 1$ and $q \geq 3$. 
\end{prop}

\begin{proof}
By \cite[Proposition 5.1]{Monod} we know that
$$
H^p_m(G,L^0((G/P)^q)) \cong 0
$$
for every $p \geq 1$ and every $q \geq 3$. %Monod's proof of  \cite[Proposition 5.1]{Monod} goes through without changes when $G$ has finitely many connected components. 
Recall from Equation \eqref{eq split exact cohomology} that
$$
H^p_m(G,L^0((G/P)^q)) \cong H^p_m(G,L^0_{\nalt}((G/P)^q)) \oplus H^p_m(G,L^0_{\alt}(G/P)^q)). 
$$
The vanishing of the left-hand side implies the statement and this concludes the proof. 
\end{proof}

\section{Proof of Theorems \ref{thm negative complex} and  \ref{thm positive complex}}

Let $G$ be a connected semisimple Lie group with finite center. Let $P$ be a minimal parabolic subgroup and fix $A<P$ a maximal split torus. We denote by $w_0$ a representative of the longest element in Weyl group. 

Before starting with the proofs of our theorems, we want to recall the explicit computation of some restriction maps in cohomology. By Monod \cite[Proposition 3.1]{Monod}, the restriction map 
\begin{equation}\label{eq res P A}
H^p_m(P) \rightarrow H^p_m(A)  
\end{equation}
is an isomorphism in all degrees. As a consequence, Monod recovers in \cite[Corollary 3.2]{Monod} a result by Wienhard \cite[Corollary 3]{Wienhard} saying that the restriction map 
\begin{equation} \label{eq res G P}
H^p_m(G) \rightarrow H^p_m(P)
\end{equation}
is trivial in all positive degrees. 

\begin{figure}[!h]
\centering
\begin{tikzpicture}
  \matrix (m) [matrix of math nodes,
             nodes in empty cells,
             nodes={minimum width=9ex,
                    minimum height=9ex,
                    outer sep=-3pt},
             column sep=2ex, row sep=-1ex,
             text centered,anchor=center]{
         p      &         &          &          & \\
          \cdots & \cdots & \cdots & \cdots   & \cdots \\
          3    & \ 0 \  & \ 0 \ &\  H^3_m(A)^{w_0} \ & 0 & \ \cdots & \\
          2    &  \ 0 \ &\  0 \  & \  H^2_m(A)^{w_0} \ & 0 & \  \cdots &  \\
          1    & \ 0 \ & \ 0 \ & \  H^1_m(A)^{w_0} \  & 0 & \   \cdots & \\
          0     & \ 0\   &\  0 \ & \  \textup{L}^0_{\nalt}((G/P)^2)^G  \ & \  \textup{L}^0_{\nalt}((G/P)^3)^G\  &  \ \cdots & \\
    \quad\strut &   0  &  1  &  2  &  3  &  \cdots & q \strut \\};

%lines are counted starting from the top 

%arrows in the 3rd line
\draw[->](m-3-4.east) -- (m-3-5.west)node[midway,above ]{$d^\rightarrow$};

%arrows in the 2nd line
\draw[->](m-4-4.east) -- (m-4-5.west)node[midway,above ]{$d^\rightarrow$};

%arrows in the 1st line
\draw[->](m-5-4.east) -- (m-5-5.west)node[midway,above ]{$d^\rightarrow$};

%arrows in the 1st line
\draw[->](m-6-4.east) -- (m-6-5.west)node[midway,above ]{$\delta$};

\draw[thick] (m-1-1.east) -- (m-7-1.east) ;
\draw[thick] (m-7-1.north) -- (m-7-7.north) ;
\end{tikzpicture}
\caption{The first page $^{II}_{\nalt}E_1$}\label{fig first page negative}
\end{figure}

\begin{figure}[!h]
\centering
\begin{tikzpicture}
  \matrix (m) [matrix of math nodes,
             nodes in empty cells,
             nodes={minimum width=9ex,
                    minimum height=9ex,
                    outer sep=-3pt},
             column sep=2ex, row sep=-1ex,
             text centered,anchor=center]{
         p      &         &          &          & \\
          \cdots & \cdots & \cdots & \cdots   & \cdots \\
          3    & \ 0 \  & \ 0 \ &\  H^3_m(A)^{w_0} \ & 0 & \ \cdots & \\
          2    &  \ 0 \ &\  0 \  & \  H^2_m(A)^{w_0} \ & 0 & \  \cdots &  \\
          1    & \ 0 \ & \ 0 \ & \  H^1_m(A)^{w_0} \  & 0 & \   \cdots & \\
          0     & \ 0\   &\  0 \ & \  H^1_{m,\nalt}(G \curvearrowright G/P)  \ & \  H^2_{m,\nalt}(G \curvearrowright G/P) \  &  \ \cdots & \\
    \quad\strut &   0  &  1  &  2  &  3  &  \cdots & q \strut \\};

%lines are counted starting from the top 

%arrows in the 3rd line
\draw[->](m-3-4) -- (m-4-6);

%arrows in the 2nd line
\draw[->](m-4-4) -- (m-5-6);

%arrows in the 1st line
\draw[->](m-5-4) -- (m-6-6);

\draw[thick] (m-1-1.east) -- (m-7-1.east) ;
\draw[thick] (m-7-1.north) -- (m-7-7.north) ;
\end{tikzpicture}
\caption{The second page $^{II}_{\nalt}E_2$} \label{fig second page negative}
\end{figure}

\begin{proof}[Proof of Theorem \ref{thm negative complex}] We start by looking at the first page of the spectral sequence $^{II}_{\nalt}E_1^{p,q}$. We first observe that the first two columns are zero. The row corresponding to $p=0$ is given by the cocomplex of $G$-invariants non-alternating cochains on $G/P$. When $p \geq 1$, by Proposition \ref{prop zero columns} the only non-trivial column corresponds to $q=2$ and by Proposition \ref{prop w0 invariant} it is isomorphic to the $w_0$-invariant cohomology of $A$, namely $^{II}_{\nalt}E_1^{p,2} \cong H^p_m(A)^{w_0}$. The first page is depicted in Figure \ref{fig first page negative}.

The second page $^{II}_{\nalt}E_2^{p,q}$ is now immediately obtained: Only the row $p=0$ is affected by the differential, and its cohomology is by definition the cohomology  $H^{q-1}_{m,\nalt}(G \curvearrowright G/P)$. For $p \geq 1$, the rows remain unchanged. We refer to Figure \ref{fig second page negative} for a picture of the second page $^{II}_{\nalt}E_2$. 

By Proposition \ref{prop first ss degenerates} we know that the spectral sequence $^I_{\nalt}E_1^{p,q}$ degenerates immediately. By \cite[Proposition A.9]{Guichardet} we must have that $^{II}_{\nalt}E_2^{p,q}$ converges to zero as well. If we denote by $d_p$ the differential of the $p$-th page, the only way that $^{II}_{\nalt}E_2$ converges to zero is when the following conditions are satisfied:
\begin{itemize}
\item the cohomology $H^p_{m,\nalt}(G \curvearrowright G/P)$ is trivial for $p=0,1,2$;
\item the differential $d_p$ determines an isomorphism between $H^{p-1}_m(A)^{w_0}$ and the cohomology $H^{p+1}_{m,\nalt}(G \curvearrowright G/P)$, when $p \geq 2$. 
\end{itemize}
This concludes the proof. 
 \end{proof}

\begin{proof}[Proof of Theorem \ref{thm positive complex}] We describe the first page of the spectral sequence $^{II}_{\alt}E_1$. The row $p=0$ is given by the cocomplex of $G$-invariants alternating cochains on $G/P$, $^{II}_{\alt}E_1^{0,q}=L^0_{\alt}((G/P)^{q})^G$. When $p \geq 1$ and $q \geq 3$, the entry $^{II}_{\alt}E_1^{p,q} \cong 0$ is trivial by Proposition \ref{prop zero columns}. We are left with the columns corresponding to $q=0,1,2$ and $p \geq 1$. When $q=0$, we get back the usual measurable cohomology of $G$, that is $^{II}_{\alt}E_1^{p,0}\cong H_m^p(G)$. For $q=1$, the Eckmann-Shapiro induction gives us back the cohomology of the minimal parabolic subgroup $P$, namely $^{II}_{\alt}E_1^{p,1} \cong H^p_m(G,L^0(G/P)) \cong H^p_m(P)$. By Equation \eqref{eq res P A} the latter is isomorphic to the measurable cohomology of the maximal split torus, that is $^{II}_{\alt}E_1^{p,1} \cong H^p_m(A)$. Finally, for $q=2$, Proposition \ref{prop w0 invariant} implies that $^{II}_{\alt}E_1^{p,q}\cong H^p_m(G,L^0_{\alt}((G/P)^2))\cong \Pi^p_{w_0}(H^p_m(A))$ is isomorphic to the $w_0$-equivariant cohomology of $A$.

\begin{figure}[!h]
\centering
\begin{tikzpicture}
  \matrix (m) [matrix of math nodes,
             nodes in empty cells,
             nodes={minimum width=9ex,
                    minimum height=9ex,
                    outer sep=-3pt},
             column sep=2ex, row sep=-1ex,
             text centered,anchor=center]{
         p      &         &          &          & \\
          \cdots & \cdots & \cdots & \cdots   & \cdots \\
          3    & \ H^3_m(G) \  & \ H^3_m(A) \ &\  \Pi^3_{w_0}(H^3_m(A)) \ & 0 & \ \cdots & \\
          2    &  \ H^2_m(G) \ &\  H^2_m(A) \  & \  \Pi^2_{w_0}(H^2_m(A)) \ & 0 & \  \cdots &  \\
          1    & \ H^1_m(G) \ & \ H^1_m(A) \ & \  \Pi^1_{w_0}(H^1_m(A)) \  & 0 & \   \cdots & \\
          0     & \ \mathbb{R}\   &\  \textup{L}^0(G/P)^G \ & \  \textup{L}^0_{\alt}((G/P)^2)^G  \ & \  \textup{L}^0_{\alt}((G/P)^3)^G\  &  \ \cdots & \\
    \quad\strut &   0  &  1  &  2  &  3  &  \cdots & q \strut \\};

%lines are counted starting from the top 

%arrows in the 3rd line
\draw[->](m-3-2.east) -- (m-3-3.west)node[midway,above ]{$0$};
\draw[->](m-3-3.east) -- (m-3-4.west)node[midway,above ]{$\Pi^3_{w_0}$};
\draw[->](m-3-4.east) -- (m-3-5.west);

%arrows in the 2nd line
\draw[->](m-4-2.east) -- (m-4-3.west)node[midway,above ]{$0$};
\draw[->](m-4-3.east) -- (m-4-4.west)node[midway,above ]{$\Pi^2_{w_0}$};
\draw[->](m-4-4.east) -- (m-4-5.west);

%arrows in the 1st line
\draw[->](m-5-2.east) -- (m-5-3.west)node[midway,above ]{$0$};
\draw[->](m-5-3.east) -- (m-5-4.west)node[midway,above ]{$\Pi^1_{w_0}$};
\draw[->](m-5-4.east) -- (m-5-5.west);

%arrows in the 1st line
\draw[->](m-6-2.east) -- (m-6-3.west)node[midway,above ]{$\delta$};
\draw[->](m-6-3.east) -- (m-6-4.west)node[midway,above ]{$\delta$};
\draw[->](m-6-4.east) -- (m-6-5.west)node[midway,above ]{$\delta$};

\draw[thick] (m-1-1.east) -- (m-7-1.east) ;
\draw[thick] (m-7-1.north) -- (m-7-7.north) ;
\end{tikzpicture}
\caption{The first page $^{II}_{\alt}E_1$}\label{fig first page positive}
\end{figure}

If we now look at the differentials, on the first row $p=0$ we simply get the usual homogeneous differential. If $p \geq 1$, it is straightforward to check that the differential 
$$d_1:^{II}_{\alt}E_1^{p,0}=H^p_m(G)\longrightarrow H^p_m(P)=^{II}_{\alt}E_1^{p,1}$$
induced by $d^\rightarrow$ is, as a map $H^p_m(G)\rightarrow H^p_m(P)$, simply $(-1)^{p+1}$ times the restriction map from $G$ to $P$. By Equation \eqref{eq res G P} we know that such a map is trivial in all degrees. As a consequence, the differentials from the column $q=0$ to the column $q=1$ vanish for $p>0$. In a similar way, the differentials from the column $q=1$ to the column $q=2$ are conjugated, up to sign, to the operators $\Pi^p_{w_0}$. More precisely, Monod \cite[Proposition 4.1]{Monod} shows that the differential can be computed at the level of cochains by 
$$
d^\rightarrow: L^0(G^{p+1})^P \rightarrow L^0(G^{p+1})^A,
$$
$$
d^\rightarrow f(g_0,\ldots,g_p)=(-1)^{p+1}\left[f(w_0^{-1}g_0,\ldots,w_0^{-1}g_p)-f(g_0,\ldots,g_p)\right]. 
$$
Since the right multiplication by $w_0$ is actually homotopic to the identity, the differential $d^\rightarrow$ is homotopic to the map $$f \mapsto (-1)^{p+1}\left[f(w_0^{-1}g_0w_0,\ldots,w_0^{-1}g_pw_0)-f(g_0,\ldots,g_p)\right],$$ which amounts to $(-1)^{p+2} \Pi^p_{w_0} \circ \mathrm{res}$ and the claim follows. 

The second page $^{II}_{\alt}E_2$ is now easy to compute. Since $H^0_m(G \curvearrowright G/P) \cong \mathbb{R}$, the inclusion $$ ^{II}E^{0,0}_1 \cong \mathbb{R} \longrightarrow L^0(G/P)^G \cong ^{II}E^{0,1}_1$$ forces the vanishing of the terms $^{II}_{\alt}E^{0,0}_2 \cong ^{II}_{\alt}E^{0,1}_2\cong 0$. When $q \geq 2$, it is clear that on the first row $p=0$ we get back $^{II}_{\alt}E^{0,q}_2 \cong H^{q-1}_{m,\alt}(G \curvearrowright G/P)$. For $p \geq 1$, on the first column we still have the cohomology of $G$. On the first page, the differentials between the first column and the second one were all surjective, thus $^{II}_{\alt}E^{p,1}_2 \cong H^p_m(A)^{w_0}$ and $^{II}_{\alt}E^{p,2}_2 \cong 0$.
We depict the second page $^{II}_{\alt}E_2$ in Figure \ref{fig second page positive}.

\begin{figure}[!h]
\centering
\begin{tikzpicture}
  \matrix (m) [matrix of math nodes,
             nodes in empty cells,
             nodes={minimum width=9ex,
                    minimum height=9ex,
                    outer sep=-3pt},
             column sep=2ex, row sep=-1ex,
             text centered,anchor=center]{
         p      &         &          &          & \\
          \cdots & \cdots & \cdots & \cdots   & \cdots \\
          3    & \ H^3_m(G) \  & \ H^3_m(A)^{w_0} \ &\  0 \ & 0 & \ \cdots & \\
          2    &  \ H^2_m(G) \ &\  H^2_m(A)^{w_0} \  & \  0 \ & 0 & \  \cdots &  \\
          1    & \ H^1_m(G) \ & \ H^1_m(A)^{w_0} \ & \  0 \  & 0 & \   \cdots & \\
          0     & \ 0\   &\  0 \ & \  H^1_{m,\alt}(G \curvearrowright G/P)  \ & \  H^2_{m,\alt}(G \curvearrowright G/P) \  &  \ \cdots & \\
    \quad\strut &   0  &  1  &  2  &  3  &  \cdots & q \strut \\};

%lines are counted starting from the top 

%arrows in the 3rd line
\draw[->](m-3-2) -- (m-4-4);
\draw[->](m-3-3) -- (m-4-5);

%arrows in the 2nd line
\draw[->](m-4-2) -- (m-5-4);
\draw[->](m-4-3) -- (m-5-5);

%arrows in the 1st line
\draw[->](m-5-2) -- (m-6-4);
\draw[->](m-5-3) -- (m-6-5);

\draw[thick] (m-1-1.east) -- (m-7-1.east) ;
\draw[thick] (m-7-1.north) -- (m-7-7.north) ;
\end{tikzpicture}
\caption{The second page $^{II}_{\alt}E_2$} \label{fig second page positive}
\end{figure}

Again by Proposition \ref{prop first ss degenerates} we know that the spectral sequence $^I_{\alt}E_1^{p,q}$ degenerates immediately. By \cite[Proposition A.9]{Guichardet} the spectral sequence $^{II}_{\alt}E_2^{p,q}$ is forced to converge to zero. If we denote by $d_p$ the differential of the $p$-th page, the only way that $^{II}_{\alt}E_2$ converges to zero is when the following conditions are satisfied:
\begin{itemize}
\item we have an isomorphism $H^0_{m,\alt}(G \curvearrowright G/P) \cong \mathbb{R}$;
\item we have an isomorphism $H^1_{m,\alt}(G \curvearrowright G/P) \cong H^1_m(G)(=0)$;
\item the differential $d_p$ sends $H^{p-1}_m(A)^{w_0}$  injectively into $H^p_{m,\alt}(G \curvearrowright G/P)$;
\item the differential $d_{p+1}$ realizes an isomorphism between $H^p_m(G)$ and the quotient $H^p_{m,\alt}(G \curvearrowright G/P)/d_p(H^{p-1}_m(A)^{w_0})$.
\end{itemize}
This proves the statement and concludes the proof. \end{proof}

\bibliographystyle{alpha}

\end{document}